\documentclass{birkjour}

\usepackage{color,colortbl,multirow}
\usepackage{overpic}

\newcolumntype{L}{>{\displaystyle}l}
\newcolumntype{C}{>{\displaystyle}c}
\newcolumntype{R}{>{\displaystyle}r}

\newcommand{\R}{\ensuremath{\mathbb{R}}}

\newcommand{\CO}{\ensuremath{\mathcal{O}}}
\newcommand{\ov}{\overline}
\newcommand{\la}{\lambda}
\newcommand{\ga}{\gamma}
\newcommand{\G}{\Gamma}
\newcommand{\T}{\theta}
\newcommand{\f}{\varphi}

\newcommand{\de}{\delta}

\def\p{\partial}
\def\e{\varepsilon}

\newtheorem {theorem} {Theorem} 
\newtheorem {proposition} [theorem] {Proposition}
\newtheorem {corollary} [theorem] {Corollary}
\newtheorem {lemma} [theorem] {Lemma}

\newtheorem {remark} {Remark}

\newtheorem {mtheorem} {Theorem}

\begin{document}
\title[]{Asymptotic behavior of periodic solutions in\\ one-parameter families of Li\'{e}nard equations}

\author[P. T. Cardin]{Pedro Toniol Cardin}

\address{Universidade Estadual Paulista (UNESP), Faculdade de Engenharia, Ilha Solteira, S\~ao Paulo, Brazil}
\email{pedro.cardin@unesp.br}

\author[D. D. Novaes]{Douglas Duarte Novaes}

\address{Universidade Estadual de Campinas (UNICAMP), Instituto de Matem\'{a}tica, Estat\'{i}stica e Computa\c{c}\~ao Cient\'{i}fica, Campinas, S\~ao Paulo, Brazil}
\email{ddnovaes@unicamp.br}



\begin{abstract}
In this paper, we consider one--parameter ($\la>0$) families of Li\'enard differential equations. We are concerned with the study on the asymptotic behavior of periodic solutions for small and large values of $\la>0$. To prove our main result we use the relaxation oscillation theory and a topological version of the averaging theory. More specifically, the first one is appropriate for studying the periodic solutions for large values of $\lambda$ and the second one for small values of $\lambda$. In particular, our hypotheses allow us to establish a link between these two theories.
\end{abstract}

\maketitle

\vspace{-1cm}
\section{Introduction and statement of the main result}\label{IntroductionSection}

In the qualitative theory of ordinary differential equations, the study of limit cycles is undoubtedly one of the main problems, which is far from trivial. Issues such as the non--existence, existence, uniqueness and other properties of limit cycles have been and continue to be studied extensively.

In particular, concerning Li\'{e}nard systems, there is a considerable amount of research on limit cycles, which began with Li\'{e}nard \cite{Lie1928}. Since the first results of Li\'{e}nard, many articles giving existence and uniqueness conditions for limit cycles of Li\'{e}nard systems have been published. See, e.g., \cite{SabVil2010,Vil1982,V83,V12}, the books \cite{Ye,Zha1992} and the references quoted therein.

In this paper, we consider one--parameter $\lambda >0$ families of Li\'enard differential equations of the form
\begin{equation}\label{Eq1}
x'' + \lambda f(x) x' + x = 0.
\end{equation}
In \eqref{Eq1} the prime indicates derivative with respect to the time $t$. Note that for $f(x) = x^2-1$ we get the classical van der Pol equation which models the oscillations of a triode vacuum tube \cite{M62,Pol1926}.

Taking $x' = y$ the differential equation \eqref{Eq1} can be converted into the first order differential system
\begin{equation}\label{S1}
x'=y, \quad y'= - x- \lambda f(x) y.
\end{equation}
In this paper, we shall assume conditions in order to ensure the existence of a unique limit cycle of system \eqref{S1} for all positive values of $\lambda.$ Accordingly, we are mainly concerned with the study on the asymptotic behavior of such limit cycle for small and large values of $\lambda>0$.

We start exposing our main hypotheses. Consider the auxiliary function
\begin{equation}\label{F}
F(x) = \int_0^x f(s) ds,
\end{equation}
and assume the following hypotheses:

\smallskip

\begin{itemize}
	\item[{\bf (A1)}] $f$ is a $C^2$--function on $\R$ having precisely two zeros, $x_M$ and $x_m$, such that $x_M<0<x_m$ and $f'(x_M)<0<f'(x_m)$. \smallskip

	\item[\textbf{(A2)}] The straight lines $y=F(x_M)$ and $y=F(x_m)$, passing through the points $A = (x_M,F(x_M))$ and $B = (x_m,F(x_m))$, intersect the graphic of $F$, $\textrm{Gr}(F)=\{(x,F(x));\,x\in \R\},$ at the points $A'\neq A$ and $B'\neq B,$ respectively. \smallskip
	
	\item[{\bf (A3)}] The differential system \eqref{S1} has at most one limit cycle.
\end{itemize}

\smallskip

We remark that hypotheses {\bf (A1)}, {\bf (A2)}, and {\bf (A3)} ensure the existence of a unique limit cycle $\Phi(\la)$ of system \eqref{S1}. The existence of $\Phi(\la)$ is a direct consequence of Dragil\"{e}v's Theorem \cite{Dragilev1952} (see Theorem \ref{DT} of Section 2). In fact, in order to get this, it is sufficient to ask the continuity of $f$, $f(0)<0$, and assume that $f(x)$ and $xF(x)$ are positive for $|x|$ large enough (see Theorem 1 from \cite{V83}).  In spite of that, we shall see that hypotheses {\bf (A1)} and {\bf (A2)} are necessary to get asymptotic informations on $\Phi(\la)$ for small and large values of $\lambda$.  

Hypothesis {\bf (A3)} is a qualitative assumption on the differential system \eqref{S1}. It is worthwhile to say that there are many analytical conditions for which it holds. Some of these conditions are provided in the Appendix (see Proposition \ref{ap:prop}).

\smallskip

Before stating our main result, we need to define some preliminary objects. Take $x_1<0<x_2$ such that $A'=(x_2,F(x_M))$ and $B'=(x_1,F(x_m))$. From {\bf (A2)} it follows that $x_1\neq x_m$ and $x_2\neq x_M$. From {\bf (A1)} and {\bf (A2)} we find exactly two zeros of $F$, $x_1^*$ and $x_2^*$, such that $x_1^*<0<x_2^*$. Clearly, $x_1<x_1^*<x_M$ and $x_m<x_2^*<x_2$. Denote
\begin{equation}\label{xr}
x^*=\min\{-x_1^*,x_2^*\}>0\quad\text{and}\quad r^*=\dfrac{\big(x_1^2+x_2^2\big)\big(F(x_M)-F(x_m)\big)}{x_1F(x_m)+x_2F(x_M)}>0.
\end{equation}

\smallskip

Let $\ov{A A'}$ and  $\ov{B B'}$ be the line segments joining $A$ to $A'$ and $B$ to $B'$, respectively. Let  $\widetilde{A'B}$ and $\widetilde{B'A}$ be the pieces of $\textrm{Gr}(F)$ joining $A'$ to $B$ and $B'$ to $A$, respectively. We define $\G_0\subset\R^2$ as being the closed curve given by the union 
\[
\G_0=\ov{A A'}\cup\widetilde{A' B}\cup\ov{BB'}\cup\widetilde{B' A}
\]
(see Figure \ref{FigIntro}(a)). We observe that the closed curve $\Gamma_0$ can be built from the hypotheses \textbf{(A1)} and \textbf{(A2)}. We note that if the hypothesis \textbf{(A2)} is not assumed, then it would not be possible to define $\Gamma_0$. Figure \ref{FigIntro}(b) illustrates a case in which the condition {\bf (A1)} is satisfied but \textbf{(A2)} is not. 

\begin{figure}[h]
	\begin{overpic}[width=5cm]{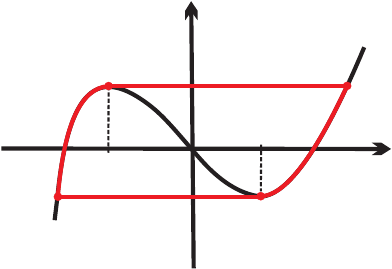}
		\put(25,50){$A$}
		\put(63,11){$B$}
		\put(91,45){$A'$}
		\put(6,16){$B'$}
		\put(65,50){$\textcolor{red}{\Gamma_0}$}
		\put(80,60){$y=F(x)$}
		\put(25,24){$x_M$}
		\put(64,34){$x_m$}
		\put(95,23){$x$}
		\put(42,63){$y$}
		\put(1,60){$(a)$}
	\end{overpic}\qquad \quad
	\begin{overpic}[width=5cm]{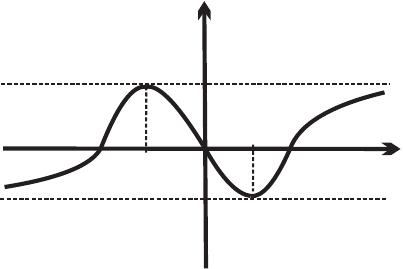}
		\put(83,35){$y=F(x)$}
		\put(95,23){$x$}
		\put(44,62){$y$}
		\put(1,60){$(b)$}
		\put(32,24){$x_M$}
		\put(60,33){$x_m$}
	\end{overpic}
	\caption{\footnotesize Panel (a) illustrates the closed curve $\Gamma_0$ and panel (b) shows an example where it is not possible to build the curve $\Gamma_0$.}
	\label{FigIntro}
\end{figure}
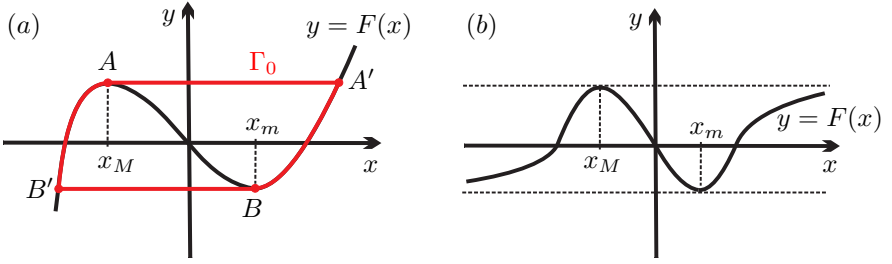

In what follows $S_\rho$ denotes the circle in $\R^2$ centered at the origin with radius equal to $\rho$ and $d_H(\cdot,\cdot)$ denotes the Hausdorff distance. Our main result is the following one.

\begin{mtheorem}\label{Thm1}
	Assume hypotheses {\bf (A1), (A2)}, and {\bf (A3)}. For each $\la>0$, let $\Phi(\la)$ denote the unique limit cycle of system \eqref{S1}. Then, the following statements hold:
	
	\begin{itemize}
		\item[(i)] The limit cycle $\Phi(\lambda )$ varies continuously on $\lambda$, for every $\lambda>0$. \vspace{.2cm}
		
		\item[(ii)] Let $P_{\lambda }(x,y)=\left(x,F(x)+y/\lambda \right)$ be an one--parameter family of diffeomorphims. Then
		\begin{equation}\label{Limit1ThmA}
		\displaystyle \lim_{\lambda \to\infty} d_H\big(\Phi(\lambda ),P_{\lambda }^{-1}(\G_0)\big)=0.
		\end{equation} 
		
		\item[(iii)] 
		There exists  $\rho\in(x^*,r^*)$ such that  
		\begin{equation}\label{Limit2ThmA}
		\displaystyle \lim_{\lambda \to 0}d_H\big(\Phi(\lambda ),S_{\rho}\big)=0.
		\end{equation}
		Furthermore, $\rho$ is the unique positive zero of the function 
		\begin{equation}\label{Fbar}
		\ov{F}(r)=\displaystyle\int_{-1}^{1}\sqrt{1-s^2}\,f(r\,s)ds.
		\end{equation}
	
		\item[(iv)] There exists $x_0>\max\{\rho,-x_1,x_2\}$ such that the limit cycle $\Phi(\la)$ is contained in the strip $\{(x,y)\in\R^2:\,-x_0<x<x_0\}$ for every $\la>0$. Moreover, the cycle $\Phi(\la)$ must intersect one of the straight lines $x=x_M$ or $x=x_m$.
		
		\item[(v)] Let $\ga:[-x_0,x_0]\times \R_+\rightarrow \R_+$ be the following continuous piecewise function 
		\begin{equation}\label{gamma}
		\ga(x;\la)=\left\{\begin{array}{ll}
		\sqrt{\dfrac{(x+x_0)(x_0-x-m\la(x+x_0))}{1+m\la}}, & 0<\la\leq-\dfrac{1}{2 m},\\
		\sqrt{(x+x_0)(x_0-x+8m^2\la^2 x_0)},& \la> -\dfrac{1}{2 m},\\
		\end{array}\right.
		\end{equation}
		where $m=\min\{f(x):\, -x_0\leq x\leq x_0\}<0$ and, for each $\la>0$,  let $R_{\la}$ denote the following compact region
		\[
		R_{\la}=\{(x,y)\in\R^2:\, -x_0\leq x\leq x_0,\, -\ga(-x;\la)\leq y\leq \ga(x;\la)\}.
		\]
		Then, $\Phi(\la)$ is contained in the interior of $R_{\la}$, for every $\la>0$ $($see Figure \ref{Rlambda}$)$.
		\end{itemize}
\end{mtheorem}

The limit given in \eqref{Limit1ThmA} means that, for sufficiently large values of $\lambda$, after a change of coordinates, the limit cycle $\Phi(\lambda )$ approaches the closed curve $\Gamma_0$. Similarly, from \eqref{Limit2ThmA} one conclude that, for sufficiently small values of $\lambda$, the limit cycle $\Phi(\lambda )$ approaches the circle $S_{\rho}$.

\begin{remark}
The proof of Theorem \ref{Thm1} involves two distinct theories, namely {\it relaxation oscillation theory} and {\it averaging theory}. At first, hypotheses  {\bf (A1)} and {\bf (A2)} were strictly conceived in order to guarantee the existence of a relaxation oscillation for system \eqref{Eq1} when $\la$ is sufficiently large. This corresponds to item (ii) of Theorem \ref{Thm1}. Roughly speaking, a relaxation oscillation is a continuous family of periodic solutions of a fast-slow system converging to a singular orbit (see Section 3 for a formal definition). Remarkably, the same hypotheses, {\bf (A1)} and {\bf (A2)}, allow the use of a topological version of the averaging theory to study the asymptotic behavior of such a limit cycle for small values of $\la>0.$ This corresponds to statement (iii) of Theorem \ref{Thm1}. Therefore, for the family of Li\'{e}nard differential equations \eqref{Eq1} under the assumptions of Theorem \ref{Thm1}, these two distinct theories provide complementary informations, which allow to study the behavior of the  limit cycle for every values of the parameter. This, in turn, establishes a link between these theories, which we understand to be the main novelty of this paper.
\end{remark}

The proof of Theorem \ref{Thm1} will be split in several results. Firstly, in Section 2, we prove the existence of the limit cycle $\Phi(\lambda )$, for every $\la>0$, and its continuous dependence on the parameter $\lambda$. Item (i) of Theorem \ref{Thm1} follows from Proposition \ref{existence}. After in Section 3, using the {\it relaxation oscillation theory}, we study the asymptotic behavior of the limit cycle $\Phi(\lambda )$ when $\lambda$ takes large values. Item (ii) of Theorem \ref{Thm1}  follows from Proposition \ref{PropRelaxationOscillation} and Corollary \ref{c1}. Similarly in Section 4, we study the asymptotic behavior of the limit cycle $\Phi(\lambda )$ when $\lambda$ takes small values, but now using mainly the {\it averaging theory}. Item (iii) of Theorem \ref{Thm1}  follows from Proposition \ref{av:prop} and Corollary \ref{c2}. Items (iv) and (v) of Theorem \ref{Thm1} are proved in Section 5, where an estimative of the amplitude growth of the limit cycle $\Phi(\lambda )$ is provided. Item (v) follows from Corollary \ref{RR}. Section 6 is devoted to study some polynomial and non-polynomial examples of differential systems where Theorem \ref{Thm1} can be applied, including the classical van der Pol equation and its generalization. In Section \ref{ConclusionSection} we present our conclusions emphasizing the link established between the relaxation oscillation theory and the averaging theory. Finally, sufficient analytical conditions in order to guarantee hypothesis {\bf (A3)} are presented in the Appendix section.

\section{Existence of the limit cycle and continuous dependence on $\lambda $}

This section is devoted to prove item (i) of Theorem \ref{Thm1}. For the sake of completeness, assuming hypotheses {\bf (A1)} and {\bf (A2)}, we shall check the Dragil\"{e}v's hypotheses for system \eqref{S1}, which   guarantee the existence of a limit cycle $\Phi(\la).$ Hypothesis {\bf (A3)} will ensure its uniqueness and continuous dependence on $\la$.

\begin{proposition}\label{existence}
Assume hypotheses {\bf (A1), (A2)}, and {\bf (A3)}. Then, for every $\lambda >0,$ the differential system \eqref{S1} has a unique stable limit cycle $\Phi(\lambda )$ which depends continuously on $\lambda $.
\end{proposition}

To prove Proposition \ref{existence} we need the next result due to Dragil\"{e}v \cite{Dragilev1952}. For a proof see Theorem 5.1 of \cite{Ye}. It is worth mentioning also that, in \cite{CioVil2015}, the authors provide an extension of Dragilev's theorem.

\begin{theorem}[Dragil\"{e}v \cite{Dragilev1952}]\label{DT}
Consider the differential system
\begin{equation}\label{DTs1}
x'=y, \quad
y'= - g(x)-  f(x) y.
\end{equation}
Suppose that:
\begin{itemize}
\item[{\bf (B1)}] The functions $F(x)$ and $g(x)$ are locally Lipschitz, where $F(x)=\int_0^xf(s)ds$.
\item[{\bf (B2)}] $x g(x)>0$ for $x\neq 0$, $G(\pm\infty)=+\infty$, where $G(x)=\int_0^xg(s)ds$; 
\item[{\bf (B3)}] There exist $a_1<0<a_2$, such that $F(x)>0$, for $a_1\leq x<0$, and $F(x)<0$, for $0<x\leq a_2$.
\item[{\bf (B4)}] There exist $k>\max\{-a_1,a_2\}$, and $b_1<b_2$ such that $F(x) \leq b_1$ if $x<-k$, and $F(x) \geq b_2$ if $x>k$.
\end{itemize}
Then, the differential system \eqref{DTs1} has at least one stable limit cycle.
\end{theorem}

\begin{proof}[Proof of Proposition \ref{existence}]
First of all we shall see that all the conditions, {\bf (B1)}--{\bf (B4)}, of Dragil\"{e}v's Theorem are fulfilled. From the hypothesis {\bf (A1)}, $F(x)$ is a differentiable function and, for the differential system \eqref{S1}, $g(x)=x$. Therefore, $F(x)$ and $g(x)$ are locally Lipschitz and then condition {\bf (B1)} is satisfied. Condition {\bf (B2)} is trivially true because $G(x)=x^2/2$. Now taking $a_1=x_M$ and $a_2=x_m$ we see that hypothesis {\bf (A2)} implies condition {\bf (B3)}. Finally condition {\bf (B4)} is assured by hypotheses {\bf (A1)} and {\bf (A2)} if we take $k>\max\{-x_1,x_2\}$, $b_1=F(x_m)$, and $b_2=F(x_M)$. 

From Dragil\"{e}v's Theorem we have assured, for each $\lambda >0$, the existence of a stable limit cycle $\Phi(\lambda )$ of the differential system \eqref{S1}. Clearly, hypothesis {\bf (A3)} implies its uniqueness for each $\lambda >0$.

Given $\lambda _0>0$ it remains to prove that for every $\e>0$ there exists $\delta>0$ such that for all $\lambda >0$ satisfying $|\lambda -\lambda _0|<\delta$ it follows that $d_H(\Phi(\lambda ),\Phi(\lambda _0))<\e$. To see that let $U_\e$ be an $\e$-neighborhood of $\Phi(\lambda _0)$, that is $d(p,q)<\e$ for every $p,q\in U_\e$. Here $d(\cdot,\cdot)$ denotes the usual distance in $\R^2$. Since $\Phi(\lambda _0)$ is a stable limit cycle we can assume that the flow of system \eqref{S1}, for $\lambda =\lambda _0$, is transversal to the boundary $\p U_\e$ of $U_\e$ and crosses $\p U_\e$ in the exterior-to-interior direction. From the continuous dependence of the solutions of \eqref{S1} on the parameter $\lambda $ we can find $\delta>0$ such that the flow of system \eqref{S1} still remains traversal to $\p U_\e$ for every $\lambda >0$ such that $|\lambda -\lambda _0|<\de$.  Therefore, for every $\lambda \in(\lambda _0-\de,\lambda _0+\de)$, $\ov{U_\e}$ is a positively invariant compact set for the flow of system \eqref{S1}. Applying now the Poincar\'{e}--Bendixon Theorem we conclude that the differential system \eqref{S1} admits a stable limit cycle $\widetilde\Phi(\lambda )\subset U_\e$ for every $\lambda \in(0,\lambda _0+\de)$. Consequently $d_H(\widetilde{\Phi}(\lambda ),\Phi(\lambda _0))<\e$. Finally hypothesis {\bf (A3)} implies that $\Phi(\lambda )=\widetilde{\Phi}(\lambda )$, and then we conclude that $\Phi$ is continuous at $\lambda =\lambda _0$. Since $\lambda _0>0$ was arbitrarily chosen we obtain the continuity of $\Phi$ for all $\lambda >0$.
\end{proof}

\section{Asymptotic behavior of the limit cycle for large values of  $\lambda$}

In this section, we will study the asymptotic behavior of the limit cycle $\Phi(\lambda )$ obtained from Theorem \ref{Thm1} for sufficiently large values of $\lambda$. For doing this we will use the relaxation oscillation theory occurring in slow--fast singularly perturbed systems. We start this section describing this theory.

Consider a two--dimensional slow--fast differential system of the form
\begin{equation}\label{SlowFast1}
\mu \dfrac{dx}{ds} = g(x,y,\mu), \qquad \dfrac{dy}{ds} = h(x,y,\mu),
\end{equation}
where $g$ and $h$ are $C^r$--functions with $r \geq 3$. 
For positive values of $\mu$, system \eqref{SlowFast1} is mutually equivalent to
\begin{equation}\label{SlowFast2}
\dfrac{dx}{d\tau} = g(x,y,\mu), \qquad \dfrac{dy}{d\tau} = \mu h(x,y,\mu),
\end{equation}
which is obtained after the time rescaling $\tau = s/\mu$.
Systems \eqref{SlowFast1} and \eqref{SlowFast2} are referred to as \textit{slow system} and \textit{fast system}, respectively. A usual way to treat with slow--fast systems is through the geometric singular perturbation theory (GSPT). The idea is to study the (limiting) fast and slow dynamics separately and then combine results on these two limiting behaviours in order to obtain information on the dynamics of the full system \eqref{SlowFast1} (or \eqref{SlowFast2}) for small values of $\mu$. 

The limiting behaviours for $\mu \to 0$ on the slow and fast time scales are given, respectively, by
\begin{equation}\label{SlowFast3}
0 = g(x,y,0), \qquad \dfrac{dy}{ds} = h(x,y,0),
\end{equation}
which will be referred to as \textit{reduced problem}, and
\begin{equation}\label{SlowFast4}
\dfrac{dx}{d\tau} = g(x,y,0), \qquad \dfrac{dy}{d\tau} = 0,
\end{equation}
which we will be referred to as \textit{layer problem}. The phase space of \eqref{SlowFast3} is the so--called \textit{critical manifold} defined by $\mathcal{S} = \{(x,y) \in \R^2 : g(x,y,0) = 0\}$. On the other hand, $\mathcal{S}$ is the set of equilibrium points for the layer problem \eqref{SlowFast4}. Among other things, Fenichel theory \cite{CT,Fen1979}  guarantees the persistence of a normally hyperbolic subset $\mathcal{S}_0 \subseteq \mathcal{S}$ as a slow manifold $\mathcal{S}_{\mu}$ of \eqref{SlowFast1} (or \eqref{SlowFast2}) for small enough values of $\mu > 0$. Moreover, the flow on $\mathcal{S}_{\mu}$ is a small perturbation of the flow of \eqref{SlowFast3} on $\mathcal{S}_0$. Normal hyperbolicity of $\mathcal{S}_0$ means that $(\partial g/\partial x) (x,y)\neq 0$ for all $(x,y) \in \mathcal{S}_0$. That is, $\mathcal{S}_0$ is normally hyperbolic if for each $ (\bar{x},\bar{y}) \in \mathcal{S}_0$, we have that $\bar{x}$ is a hyperbolic equilibrium point of $(dx/d\tau) = g(x,\bar{y},0)$. 

Generically, a non--normally hyperbolic point is a fold point of $\mathcal{S}$. In the case when the critical manifold $\mathcal{S}$ has non--normally hyperbolic points, interesting global phenomena can occur. For instance, an interesting kind of global phenomenon is the relaxation oscillations. A \textit{relaxation oscillation} is a periodic solution $\Gamma_{\mu}$ of the slow--fast system \eqref{SlowFast1} that converges to a singular trajectory $\Gamma_0$, when $\mu \to 0$, with respect to Hausdorff distance. A singular trajectory means a curve obtained as concatenations of trajectories of the reduced and layer problems (with a consistent orientation) forming a closed curve.


In what follows we describe a well known prototypical situation where a relaxation oscillation exists (see \cite{KruSzm2001,MR80,Pon57}). 



\begin{itemize}
	\item[\textbf{(C1)}] The critical manifold $\mathcal{S}$ can be written in the form $y=\phi(x)$ and the function $\phi$ has precisely two critical points, one minimum $\ov x_m$ and one maximum $\ov x_M$, both non--degenerate (folds). \vspace{.2cm}
	
	\item[\textbf{(C2)}] The fold points are generic, i.e. 
	$$\dfrac{\partial^2 g}{\partial x^2}(\ov x_m,\phi(\ov x_m),0) \neq 0, \quad \dfrac{\partial g}{\partial y}(\ov x_m,\phi(\ov x_m),0) \neq 0, \quad h(\ov x_m,\phi(\ov x_m),0) \neq 0,$$ $$\dfrac{\partial^2 g}{\partial x^2}(\ov x_M,\phi(\ov x_M),0) \neq 0, \quad \dfrac{\partial g}{\partial y}(\ov x_M,\phi(\ov x_M),0) \neq 0, \quad h(\ov x_M,\phi(\ov x_M),0) \neq 0.$$
	
	\vspace{.2cm}
	
	\item[\textbf{(C3)}] $(\partial g/\partial x) < 0$ on $\mathcal{S}_l = \{(x,\phi(x)) : x < \ov x_m\}$ and on $\mathcal{S}_r = \{(x,\phi(x)) : x > \ov x_M\}$, and $(\partial g/\partial x)  > 0$ on $\mathcal{S}_m = \{(x,\phi(x)) : \ov x_m < x < \ov x_M\}$. This means that for the layer problem \eqref{SlowFast4} the branches $\mathcal{S}_l$ and $\mathcal{S}_r$ are attracting while $\mathcal{S}_m$ is repelling. \vspace{.2cm}
	
	\item[\textbf{(C4)}] The slow (reduced) flow on $\mathcal{S}_l$ and $\mathcal{S}_r$ satisfies that $(dy/ds) < 0$ and $(dy/ds)  > 0$, respectively.\vspace{.2cm} 
\end{itemize}
Figure \ref{fig1}(a) illustrates the phase portraits of the reduced and layer problems \eqref{SlowFast3} and \eqref{SlowFast4}  (assuming the hypotheses \textbf{(C1)}--\textbf{(C4)}).

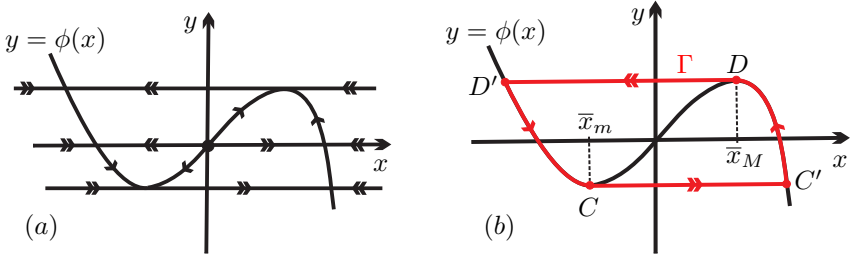
\begin{figure}
\begin{overpic}[width=5cm]{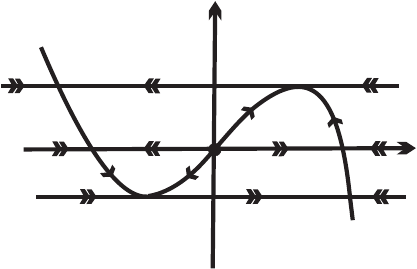}
\put(-2,55){$y=\phi(x)$}
\put(95,22){$x$}
\put(45,60){$y$}
\put(3,5){$(a)$}
\end{overpic}\qquad \quad
\begin{overpic}[width=5cm]{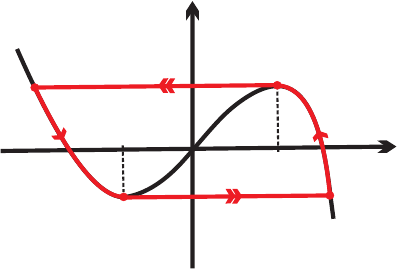}
\put(-1,42){$D'$}
\put(85,17){$C'$}
\put(67,48){$D$}
\put(28,10){$C$}
\put(54,48){$\textcolor{red}{\Gamma}$}
\put(-7,57){$y=\phi(x)$}
\put(28,33){$\ov x_m$}
\put(67,24){$\ov x_M$}
\put(95,23){$x$}
\put(42,62){$y$}
\put(3,5){$(b)$}
\end{overpic}
\caption{\footnotesize Assuming the hypotheses \textbf{(C1)}--\textbf{(C4)}, panel (a) illustrates the phase portraits of the reduced and layer problems \eqref{SlowFast3} and \eqref{SlowFast4}. Note that double and single arrows indicate the direction of the fast and slow flows, respectively. Panel (b) illustrates the singular trajectory $\Gamma$.}
\label{fig1}
\end{figure}

Let be $C = (\ov x_m,\phi(\ov x_m))$ and $D = (\ov x_M,\phi(\ov x_M))$. Let  $C'$ and $D'$ be the points of intersection of the straight lines $y = \phi(\ov x_m)$ and $y = \phi(\ov x_M)$ with $\mathcal{S}_r$ and $\mathcal{S}_l$, respectively. Consider $\Gamma$ the singular trajectory defined as the union of the fast fibers joining $C$ to $C'$ and $D$ to $D'$ and of the two pieces of the critical manifold $\mathcal{S}$ joining $C'$ to $D$ and $D'$ to $C$. See Figure \ref{fig1}(b). Let $V$ be a small tubular neighborhood of $\Gamma$. Then, under the assumptions \textbf{(C1)}--\textbf{(C4)}, for $\mu$ small enough, system \eqref{SlowFast1} admits a unique stable limit cycle $\Gamma_{\mu} \subset V$ which converges to the singular trajectory $\Gamma$ in the Hausdorff distance as $\mu \to 0$ (see \cite{KruSzm2001,MR80,Pon57}).

Next result states that, under the conditions {\bf (A1)} and \textbf{(A2)}, the differential system \eqref{S1} has, after a change of variables, a unique stable limit cycle approaching to the singular trajectory $\Gamma_0$ as $\lambda \to \infty$. 

\begin{proposition}\label{PropRelaxationOscillation}
	Assume that the hypotheses {\bf (A1)} and \textbf{(A2)} are fulfilled. Then, there exists $\lambda _1>0$ such that the differential system \eqref{S1} admits a stable limit cycle $\Phi_1(\lambda)$ for every $\lambda \in(\lambda _1,+\infty)$. Moreover
	$$\displaystyle \lim_{\lambda \to \infty} d\big(\Phi_1(\lambda),P_{\lambda}^{-1}(\Gamma_0)\big)=0$$ 
	where $P_{\lambda }(x,y)=\left(x,F(x)+y/\lambda \right)$  is an one--parameter family of diffeomorphims, and $\Gamma_0$ is the singular trajectory defined in Section \ref{IntroductionSection}, see Figure \ref{fig2}(a).
\end{proposition}

Clearly, hypothesis {\bf (A3)} implies that $\Phi(\lambda )=\Phi_1(\lambda ),$ for every $\lambda \in(\lambda _1,+\infty)$. Nevertheless, assuming weaker hypotheses (continuity of $f$; $f(0)<0$;  $f(x)>0$ and $xF(x)>0$, for $|x|$ large enough), it is known (see Theorem 2.2 from \cite{V12}) that there exists $\hat \la>0$ such that system \eqref{S1} has a unique limit cycle for every $\la>\hat\la$. Consequently, as a trivial consequence of Proposition \ref{PropRelaxationOscillation} we obtain the following result.

\begin{corollary}\label{c1}
Assume that the hypotheses {\bf (A1)} and  \textbf{(A2)} are fulfilled. Then, $\Phi(\lambda )=\Phi_1(\lambda )$ for every $\lambda \in(\lambda _1,+\infty)$.
\end{corollary}

\begin{proof}[Proof of Proposition \ref{PropRelaxationOscillation}]
	Firstly, we transform system \eqref{S1} into a slow--fast system. To do this we first define a new independent variable $s$ setting $s := t/\lambda$. This leads to system $\dot{x} = \lambda y, \dot{y}= - \lambda x- \lambda^2 f(x) y,$
	where the dot means derivative with respect to $s$. After that, we apply the change of coordinates $(x,u) = P_{\lambda}(x,y) = (x,F(x)+y/\lambda)$, where $F$ is given in \eqref{F}. Then, we obtain the system $(1/\lambda^2)\dot{x}= u - F(x), \dot{u}= -x.$
	Finally, setting $\mu := 1/\lambda^2$ and considering $\lambda$ large enough, we get the following slow--fast singularly perturbed system with small perturbation parameter $\mu$:
	\begin{equation}\label{Eq2}
	\mu\dot{x} = u - F(x), \qquad \dot{u} = -x.
	\end{equation}
	Comparing system \eqref{Eq2} with the general form \eqref{SlowFast1}, we have that $g(x,u,\mu) = u - F(x)$ and $h(x,u,\mu) = -x$.
	The critical manifold is given by $\mathcal{S} = \{(x,u) \in \R^2 : u = F(x)\}$. From the hypothesis {\bf (A1)}, it follows that $F$ has precisely two critical points, $x_M$ and $x_m$, since $F'(x) = f(x)$ for all $x \in \R$, and $x_M$ and $x_m$ are the only zeros of $f$. Moreover, as $F''(x_M) = f'(x_M) < 0$ and $F''(x_m) = f'(x_m) > 0$, then $A = (x_M,F(x_M))$ and $B = (x_m,F(x_m))$ are maximum and minimum points of $F$, respectively, both non--degenerate. Also these two fold points are generic, since
	$$\dfrac{\partial^2 g}{\partial x^2}(A) = -F''(x_M) \neq 0, \quad \dfrac{\partial g}{\partial u}(A) = 1 \neq 0, \quad h(A) = -x_M \neq 0,$$ $$\dfrac{\partial^2 g}{\partial x^2}(B) = -F''(x_m) \neq 0, \quad \dfrac{\partial g}{\partial u}(B) = 1 \neq 0, \quad h(B) = -x_m \neq 0.$$
	Further, from the hypothesis {\bf (A1)}, we can conclude that $F'(x) < 0$ for $x \in (x_M, x_m)$, and $F'(x) > 0$ for $x \in (-\infty, x_M) \cup (x_m,+\infty)$. Therefore, for the layer problem given by $\dot{x} = u - F(x), \dot{u} = 0$, the branches $\mathcal{S}_l = \{(x,F(x)) : x < x_M\}$ and  $\mathcal{S}_r = \{(x,F(x)) : x > x_m\}$ are attracting while the branch $\mathcal{S}_m = \{(x,F(x)) : x_M < x < x_m\}$ is repelling. Relative to dynamics of the reduced problem, we have that the slow flow on $\mathcal{S}_l$ satisfies $\dot{u} > 0$ and on $\mathcal{S}_r$ it satisfies that $\dot{u} < 0$. On $\mathcal{S}_m$ the reduced problem has a repeller equilibrium at the point $(0,F(0))=(0,0)$. In short, from the hypothesis {\bf (A1)}, we have that the fast and slow dynamics are as illustrated in Figure \ref{fig2}(b).
	
	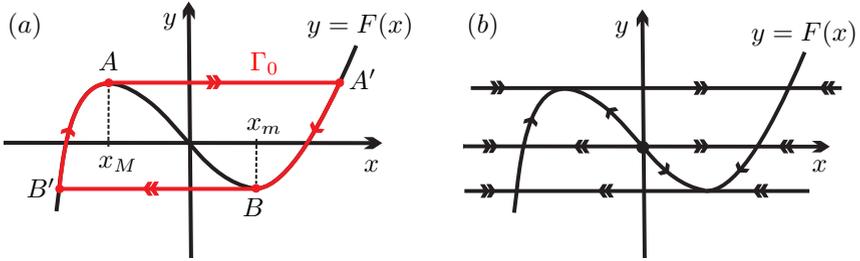
\begin{figure}
		\begin{overpic}[width=5cm]{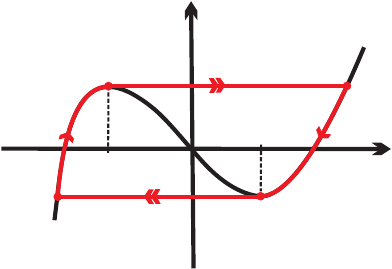}
			\put(25,50){$A$}
			\put(63,11){$B$}
			\put(91,45){$A'$}
			\put(6,16){$B'$}
			\put(65,50){$\textcolor{red}{\Gamma_0}$}
			\put(80,60){$y=F(x)$}
			\put(25,24){$x_M$}
			\put(64,34){$x_m$}
			\put(95,23){$x$}
			\put(42,62){$y$}
			\put(1,60){$(a)$}
		\end{overpic}\qquad \quad
		\begin{overpic}[width=5cm]{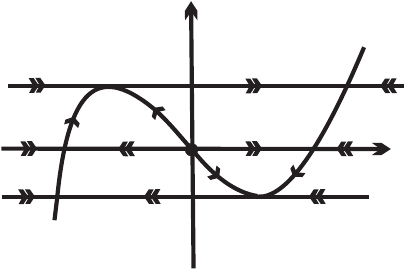}
			\put(76,58){$y=F(x)$}
			\put(92,23){$x$}
			\put(40,60){$y$}
			\put(1,60){$(b)$}
		\end{overpic}
		\caption{\footnotesize Panel (a) illustrates the singular trajectory $\Gamma_0$ and panel (b) illustrates the phase portraits of the layer and reduced problems associated with system \eqref{Eq2}. The double and single arrows indicate the direction of the fast and slow flows, respectively.}
		\label{fig2}
	\end{figure}
	
	We remark that the situation presented above involving system \eqref{Eq2} is not exactly the same as in prototypical situation described from the assumptions \textbf{(C1)}--\textbf{(C4)}. In this last case the critical manifold $\mathcal{S}$ is S--shaped while in our case $\mathcal{S}$ has the shape of a S but reflected. Mathematically speaking, the two situations coincide after a reflection $(X,Y)=(-x,y)$ on the $y$-axis.
	
	 Let $U$ be a small tubular neighborhood of $\Gamma_0$. Then, under the hypotheses {\bf (A1)} and \textbf{(A2)} it follows that, for $\mu$ small enough, system \eqref{Eq2} has a unique stable limit cycle $\Gamma_{\mu} \subset U$ which converges to the singular trajectory $\Gamma_0$ in the Hausdorff distance as $\mu \to 0$. Consequently, using the reverse change of coordinates, we can conclude that, for $\lambda > 0$ sufficiently large, there exists a unique stable limit cycle $\Phi_1(\lambda) \subset U$ for system \eqref{S1} such that
	 $$\displaystyle \lim_{\lambda \to \infty} d\big(\Phi_1(\lambda),P_{\lambda}^{-1}(\Gamma_0)\big)=0.$$
	This completes the proof of Proposition \ref{PropRelaxationOscillation}.
\end{proof}

\section{Asymptotic behavior of the limit cycle for small values of $\lambda $}

In this section, we shall use the averaging theory to study the behavior of the limit cycle $\Phi(\lambda )$ for small values of $\lambda >0$. For a general introduction on this subject we refer the book of Sanders, Verhulst and Murdock \cite{SVM}. The next theorem is a topological version of the classical Averaging Theorem to find periodic orbits for smooth differential systems.

\smallskip

\begin{theorem}[Averaging Theorem]\label{at} Consider the following nonautonomous differential equation
\begin{equation}\label{at:s1}
\dfrac{d r}{d\T}=\e F_1(\T,r)+\e^2 R(\T,r;\e),
\end{equation}
where, for $\ov r>0$ and $\e_0>0$ a small parameter, $F_1:\R\times(0,\ov r)\rightarrow\R$ and $R:\R\times(0,\ov r)\times(-\e_0,\e_0)\rightarrow\R$ are continuous functions, $2\pi$-periodic in the first variable, and locally Lipschitz in the second variable. We define the averaged function $M_1:(0,\ov r)\rightarrow \R$ as
\begin{equation}\label{at:f1}
M_1(r)=\int_0^{2\pi} F_1(\T,r)d\T.
\end{equation}
Assume that for some $\rho\in(0,\ov r)$ with $M_1(\rho)=0$, there exists a neighborhood $V$ of $\rho$ such that $M_1(r)\neq 0$ for all $r\in \ov V\setminus\{\rho\}$ and $d_B(M_1,V,0)\neq0$. Then, for $|\e|\neq0$ sufficiently small, there exists a $2\pi$-periodic solution $\f(\T;\e)$ of the differential equation \eqref{at:s1} such that $\f(0;\e)\to \rho$ when $\e\to 0$.
\end{theorem}

\smallskip

The above theorem was proved in \cite{BL} and then generalized in \cite{LNT}. The function $d_B(M_1,V,0)$ denotes the Brouwer degree of the function $M_1$ with respect to the domain $V$ and the value $0$ (see \cite{B} for a general definition).  When $M_1$ is a $C^1$--function and the Jacobian determinant of  $M_1$ at
$r\in V$ is distinct from zero (we denote $J_{M_1}(r)\neq 0)$  then the Brouwer degree of $M_1$ at $0$ is given by
\begin{equation}\label{defbr}
d_B(M_1,V,0)=\sum_{r\in\mathbf{Z}_{M_1}}\mathrm{sign}\left(J_{M_1}(r)\right),
\end{equation}
where $\mathbf{Z}_{M_1}=\lbrace r\in V :M_1(r)=0\rbrace$. In this case $J_{M_1}(\rho)\neq0$ implies $\displaystyle d_B\left(
M_1,V,0\right)=1$ for some small neighborhood $V$ of $\rho$. 

The main property of the Brouwer degree we shall use in this section is the {\it invariance under homotopy} (see \cite{B}), which says:

\smallskip

\noindent{\bf Invariance under homotopy.} {\it Let $M_s(r)$ be an homotopy between $M_0$ and $M_1$ for $s\in[0,1]$. If $0\notin M_s(\p V)$ for every $s\in[0,1]$, then $d_B(M_s,V,0)$ is constant in $s$.}

\smallskip

The proof of  Theorem \ref{at} (Averaging Theorem) is based on the fact that the {\it Poincar\'{e} map} $\Pi:\Sigma\rightarrow\Sigma$ of the nonautonomous differential equation \eqref{at:s1} defined on the {\it Poincar\'{e} section} $\Sigma=\{(\T,r)\in\R\times (0,\ov r):\,\T=0\}$ reads
\begin{equation}\label{poincare}
\Pi(r;\e)=r+\e M_1(r)+\CO(\e^2).
\end{equation}
Let $\phi(\theta;\e)$ be a family of periodic solutions of the differential equation \eqref{at:s1}. Thus, $r(\e)=\phi(0;\e)$ is a branch of fixed points of the Poincar\'{e} map $\Pi(r;\e)$. From \eqref{poincare} we have that $M_1(r(0))=0$, that is the branch $r(\e)$ approaches to the set of zeros of $M_1$. The {\it degree theory} allows us to provide sufficient conditions for which the conversely is true, assuring when a zero $\rho$ of $M_1(r)$ will persist as branch of fixed points $r(\e)$ of $\Pi(r;\e)$, that is $\Pi(r(\e);\e)=r(\e)$ and $r(0)=\rho$. 

The stability of a periodic solution $\phi(t;\e)$ associated with a branch of fixed points $r(\e)=\phi(0;\e)$ of the Poincar\'{e} map $\Pi(r;\e)$ can be studied via the  {\it displacement function}, which is defined as $\Delta(r;\e)=\Pi(r;\e)-r$. For a fixed $\e>0$ small enough the periodic solution $\phi(r;\e)$ is
\begin{itemize}
\item[(i)] unstable (or repelling) if there exists a small neighborhood $I=(a,b)$ of $r(\e)$ such that $\Delta(r;\e)>0$ for every $r\in(r(\e),b)$ and $\Delta(r;\e)<0$ for every $r\in(a,r(\e))$;
\item[(ii)] stable (or attracting)  if there exists a small neighborhood $I=(a,b)$ of $r(\e)$ such that $\Delta(r;\e)<0$ for every $r\in(r(\e),b)$ and $\Delta(r;\e)>0$ for every $r\in(a,r(\e))$.
\end{itemize}
We shall see that the expression \eqref{poincare} also help us to study the stability of a periodic solution given by Theorem \ref{at}.

Next result states that, under the conditions {\bf (A1)},  \textbf{(A2)}, and \textbf{(A3)}, the differential system \eqref{S1} has a unique stable limit cycle approaching to the circle $S_{\rho}$ as $\lambda \to 0$. Before stating it we recall the definitions of $x^*$ and $r^*$:
\[
x^*=\min\{-x_1^*,x_2^*\},\quad\text{and}\quad r^*=\dfrac{\big(x_1^2+x_2^2\big)\big(F(x_M)-F(x_m)\big)}{x_1F(x_m)+x_2F(x_M)}>x^*,
\]
where $x_1<0<x_2$ are the abscissas of $A'$ and $B'$, respectively,  and $x_1^*<0<x_2^*$ are the unique zeros of $F$ distinct from zero. As mentioned before $x_1<x_1^*<x_M$ and $x_m<x_2^*<x_2$.
\begin{proposition}\label{av:prop}
Assume that the hypotheses {\bf (A1)},  \textbf{(A2)}, and {\bf (A3)} are fulfilled. Then, there exists $\lambda _0>0$ such that the differential system \eqref{S1} admits a stable limit cycle $\Phi_0(\lambda )$ for every $\lambda \in(0,\lambda _0)$. Moreover, there exists $\rho\in(x^*,r^*)$ such that 
\begin{equation}\label{limit2}
\displaystyle \lim_{\lambda \to 0}d\big(\Phi_0(\lambda ),S_{\rho}\big)=0,
\end{equation}
being $\rho$ the unique zero of the function $\ov F(r)$ defined in \eqref{Fbar}.
\end{proposition}

As a trivial consequence of Proposition \ref{av:prop} and hypothesis {\bf (A3)} we obtain the following result.

\begin{corollary}\label{c2}
Assume that the hypotheses {\bf (A1)},  \textbf{(A2)}, and {\bf (A3)} are fulfilled. Then, $\Phi(\lambda )=\Phi_0(\lambda )$ for every $\lambda \in(0,\lambda _0)$.
\end{corollary}

In order to prove Proposition \ref{av:prop} we shall need the next lemma.
\begin{lemma}\label{av:lemma}
The following inequality holds for every $u>1$:
\begin{equation}\label{ine2}
\dfrac{1}{u}<\arctan\left(\dfrac{1}{\sqrt{u^2-1}}\right)<\dfrac{\pi}{2 u}.
\end{equation}
\end{lemma}

\begin{proof}
First we note that the inequality \eqref{ine2} holds for $u=2$, since $1/2<\pi/6<\pi/4$. Now consider the functions
\begin{equation}\label{beta}
\begin{array}{l}
\beta_1(u)=\dfrac{1}{u}-\arctan\left(\dfrac{1}{\sqrt{u^2-1}}\right),\vspace{0.2cm}\\

\beta_2(u)=\arctan\left(\dfrac{1}{\sqrt{u^2-1}}\right)-\dfrac{\pi}{2 u}.
\end{array}
\end{equation}
The following properties hold:
\begin{equation}\label{lbeta1}
\lim_{u\to 1}\beta_1(u)=1-\dfrac{\pi}{2}<0,\quad \lim_{u\to +\infty}\beta_1(u)=0, 
\end{equation}
\begin{equation}\label{lbeta2}
\lim_{u\to 1}\beta_2(u)= \lim_{u\to +\infty}\beta_2(u)=0. 
\end{equation}
Computing the derivatives of the functions $\beta_1$ and $\beta_2$ given in \eqref{beta} we get
\begin{equation}\label{dbeta}
\beta_1'(u)=\dfrac{u-\sqrt{u^2-1}}{u^2\sqrt{u^2-1}},
\quad
\beta_2'(u)=\dfrac{\pi\sqrt{u^2-1}-2u}{2u^2\sqrt{u^2-1}}.
\end{equation}

From the limiting values \eqref{lbeta1}, if the function $\beta_1$ has a zero $u^*>1$ then it must have at least one critical point for $u>1$. However from \eqref{dbeta} we know that $\beta_1'(u)>0$ for every $u>1$. Hence, we conclude that $\beta_1(u)<0$ for every $u>1$, which leads to the first inequality of \eqref{ine2}.

From the limiting values \eqref{lbeta2}, if the function $\beta_2$ has a zero $u^*>1$ then it must have at least two critical points for $u>1$. However from \eqref{dbeta} we know that $\beta_2'$ vanishes only for $u=\pi/\sqrt{\pi^2-4}>1$. Since $\beta_2(2)<0$ we conclude that $\beta_2(u)<0$ for every $u>1$, which leads to the second inequality of \eqref{ine2}. This completes the proof of Lemma \ref{av:lemma}. 
\end{proof}

\begin{proof}[Proof of Proposition \ref{av:prop}]
In order to write the differential system \eqref{S1} in the standard form \eqref{at:s1} of the Averaging Theorem, 
we transform it through the coordinates changing $x=r \cos\T$ and $y=-r \sin\T$. The transformed system reads
\begin{equation}
r'=-\lambda\,r\, f(r\cos\T)\sin^2\T ,\quad \T'=1-\lambda f(r\cos\T)\sin\T\cos\T.
\end{equation}
We note that $\T'>0$ for $\lambda >0$ sufficiently small. Hence, we can take the angle $\T$ as the new time variable, that is
\begin{equation}\label{polar}
\dfrac{d r}{d\T}=\dfrac{-\lambda\,r\,f(r\cos\T)\sin^2\T}{1-\lambda\,f(r\cos\T)\sin\T\cos\T}=-\lambda \,r\, f(r \cos\T)\sin^2\T+\lambda ^2 R(\T,r;\lambda ).
\end{equation}
Therefore, the differential system \eqref{S1} is equivalent to the differential equation \eqref{polar} which is written in the standard form \eqref{at:s1} with $\e=\lambda$. Moreover, since $f$ is continuous, $F$ is differentiable and then the right handside of the differential equation \eqref{polar} is locally Lipschitz in the variable $r$. 

Computing the averaged function \eqref{at:f1} for the differential equation \eqref{polar} we obtain
\begin{equation}\label{prom1}
M_1(r)= -r\displaystyle\int_0^{2\pi}      f(r \cos\T)\sin^2\T \,d\T= -2 r\ov F(r).
\end{equation}
To get the second above equality firstly we use the change of variable $x=r\cos\T$, restricted to the domains $[0,\pi/2]$, $[\pi/2,\pi]$, $[\pi,3\pi/2]$, and $[3\pi/2,2\pi]$, and then we take $x=r\,s$.  Equivalent formulae for the averaged function \eqref{prom1} are given by
\begin{equation}\label{prom}
M_1(r)=\!-\displaystyle\int_0^{2\pi}\!\!\!\!\cos\T\, F(r \cos\T)d\T
=-\dfrac{2}{r}\displaystyle\int_{-r}^{r}\!\!\dfrac{x}{\sqrt{r^2-x^2}}F(x)dx.
\end{equation}
The first above equality can be checked using integration by parts. The second one is also obtained using the change of variable $x=r\cos\T$ restricted to the domains $[0,\pi/2]$, $[\pi/2,\pi]$, $[\pi,3\pi/2]$, and $[3\pi/2,2\pi].$ Throughout this proof the last equality of \eqref{prom} will be more conveniently for our purposes.

Since $x F(x)\leq0$ for $x\in(x_1^*,x_2^*)$ and $\sqrt{r^2-x^2}>0$ for every $x\in(-r,r)$, we obtain a first estimative for the averaged function \eqref{prom}:  $M_1(r)>0$ for every $r\in[0,x^*]$.

From the hypotheses {\bf (A1)} and {\bf (A2)} we have the following properties:
\begin{itemize}
\item[{\bf(p$_1$)}] $xF(x)>x_1 F(x_m)$ for every $x<x_1$;
\item[{\bf(p$_2$)}] $xF(x)>x_1 F(x_M)$ for every $x_1<x<0$;
\item[{\bf(p$_3$)}] $xF(x)>x_2 F(x_m)$ for every $0<x<x_2$;
\item[{\bf(p$_4$)}] $xF(x)>x_2 F(x_M)$ for every $x>x_2$.
\end{itemize}

To obtain {\bf (p$_1$)} we note that $F(x)<F(x_m)$ for $x<x_1<0$, therefore  $x F(x)>xF(x_m)$. Moreover, $F(x_m)<0$ and $x<x_1$ imply that $x F(x_m)>x_1 F(x_m)$, which leads to {\bf (p$_1$)}. The other properties follow using similar arguments.

If $r>\max\{-x_1,x_2\}$ the last integral in \eqref{prom} can be split in the domains $I_1=[-r,x_1]$, $I_2=[x_1,0]$, $I_3=[0,x_2]$, and $I_4=[x_2,r]$. 
The property {\bf (p$_i$)} can be used to estimate the integral \eqref{prom} restricted to $I_i$, $i=1,2,3,4$. For instance, using {\bf (p$_{1}$)} on the domain $I_1$ we get
\[
\begin{array}{rl}
\displaystyle\int_{-r}^{x_1}\dfrac{x}{\sqrt{r^2-x^2}}F(x)dx>&\displaystyle x_1 F(x_m)\int_{-r}^{x_1}\dfrac{1}{\sqrt{r^2-x^2}}dx\vspace{0.2cm}\\
=&\dfrac{x_1 F(x_m)}{2}\left(\pi+2\arctan\left(\dfrac{x_1}{\sqrt{r^2-x_1^2}}\right)\right).
\end{array}
\]
Doing the same for $i=2,3,4$ we get the following inequality
\begin{equation}\label{1M1}
\begin{array}{rl}
\displaystyle M_1(r)\leq&-\dfrac{\pi\big(x_1F(x_m)+x_2 F(x_M)\big)}{r}\vspace{0.2cm}\\
&+\dfrac{2\,x_1\big(F(x_m)-F(x_M)\big)}{r}\arctan\left(\dfrac{-x_1}{\sqrt{r^2-x_1^2}}\right)\vspace{0.2cm}\\
&-\dfrac{2\,x_2\big(F(x_m)-F(x_M)\big)}{r}\arctan\left(\dfrac{x_2}{\sqrt{r^2-x_2^2}}\right).
\end{array}
\end{equation}
From Lemma \ref{av:lemma} we know that the following inequality holds for $0<a<r$:
\begin{equation}\label{ine}
\arctan\left(\dfrac{a}{\sqrt{r^2-a^2}}\right)\leq\dfrac{\pi a}{2 r}.
\end{equation}
Note that, in \eqref{1M1}, the coefficients of the arctangents are all positive. Thus, applying \eqref{ine} into \eqref{1M1} we obtain the following second estimative for the averaged function \eqref{prom}:
\[
\begin{array}{rl}
\displaystyle M_1(r)\leq\dfrac{\pi}{r^2}\bigg[\big(x_1^2+x_2^2\big)\big(F(x_M)-F(x_m)\big)-r\big(x_1F(x_m)+x_2F(x_M)\big)\bigg].
\end{array}
\]
Let $r^*$ be the zero of the righthand side of the above inequality, that is
\[
r^*=\dfrac{\big(x_1^2+x_2^2\big)\big(F(x_M)-F(x_m)\big)}{x_1F(x_m)+x_2F(x_M)}>x^*.
\]
Since $x_1F(x_m)+x_2F(x_M)>0$ it follows that $M_1(r)<0$ for every $r>r^*$. Hence, we conclude that there exists $\rho\in(x^*,r^*)$ such that $M_1(\rho)=-2 \rho \ov F(\rho)=0$. Moreover, $M_1(r)<0$ for $r>\rho$ and  $M_1(r)>0$ for $0<r<\rho$, and therefore we conclude that $\rho$ is the unique zero of $M_1$ and consequently of  $\ov F$.

Now consider the homotopy $M_s(r)=(1-s)(\rho-r)+s M_1(r),$ $s\in[0,1],$ between $M_0(r)=\rho-r$ and $M_1(r)$. Clearly, for every $s\in[0,1],$ $M_s(\rho)=0$, $M_s(r)<0$ for $r>\rho$ and $M_s(r)>0$ for $r<\rho$. From \eqref{defbr} it is easy to see that $d_B(M_0,V,0)=-1$. Therefore, from the {\it invariance under homotopy} property we conclude that $d_B(M_1,V,0)=-1$.

Applying the Averaging Theorem we conclude that, for $\lambda >0$ sufficiently small, the differential equation \eqref{polar} has a periodic solution $\phi(\T;\lambda )$ such that $\phi(0;\lambda )\to\rho$ when $\lambda \to 0$. Since the solutions of the differential equation \eqref{polar} for $\lambda =0$ are constant in the variable $\T$ we conclude that $\phi(\T;\lambda )\to\rho$ when $\lambda \to 0$ for every $\T\in[0,2\pi].$  Consequently, for $\lambda >0$ sufficiently small, the closed curve 
\[
\Phi_0(\lambda )=\Big\{\big(\phi(\T;\lambda )\cos\T,-\phi(\T;\lambda )\sin\T\Big):\, \T\in[0,2\pi]\Big\}
\]
describes a limit cycle of system \eqref{S1} such that $\lim_{\lambda \to 0}d\big(\Phi_0(\lambda),S_{\rho}\big)$ $=0$. 

Let $r(\lambda )=\phi(0;\lambda )$ be the branch of fixed points of the Poincar\'{e} map $\Pi(r;\lambda )$. From the hypothesis {\bf (A3)} this branch is unique. Therefore, from \eqref{poincare}
\[
\Delta(r;\lambda )= \Pi(r;\lambda )-r=\lambda M_1(r)+\CO(\lambda^2),
\]
and then, for $\lambda >0$ sufficiently small, $\Delta(r;\lambda )<0$ for every $r>r(\lambda )$ and $\Delta(r;\lambda )>0$ for every $0<r<r(\lambda )$. This implies that the periodic solution $\phi(\T;\lambda )$ and, consequently, the limit cycle $\Phi_0(\lambda )$ is stable.
\end{proof}

\section{Limit cycle amplitude growth}

We observe that, from item (iii) of Theorem \ref{Thm1}, the amplitude of the limit cycle $\Phi(\la)$ tends to $\rho$ when $\la$ goes to $0$, and from item (ii) of Theorem \ref{Thm1}, the amplitude of $\Phi(\la)$ is unbounded for $\la>0$. In this section, we aim to investigate the amplitude growth of $\Phi(\la)$.  To do that we build a region $R_{\la}$ having an increasing diameter such that $\Phi(\la)\subset R_{\la}$, for every $\la>0$ (see Figure \ref{Rlambda}). Before that we provide the proof of item (iv) of Theorem \ref{Thm1}.  

\begin{proof}[Proof of item (iv) of Theorem \ref{Thm1}]
Firstly, consider the functions 
\[
\xi^-(\la)=\min\big(\pi_x\Phi(\la)\big)<0\quad \text{and}\quad \xi^+(\la)=\max\big(\pi_x\Phi(\la)\big)>0,
\]
where $\pi_x$ denotes the projection onto the axis $x$. Item (i) of Theorem \ref{Thm1} implies that $\xi^{\pm}(\la)$ are continuous for every $\la>0$. Furthermore, items (ii) and (iii) of Theorem \ref{Thm1} imply
\[
\lim_{\la\to0}\xi^{\pm}(\la)=\pm\rho, \quad \lim_{\la\to\infty}\xi^{-}(\la)=x_1, \quad \text{and}\quad \lim_{\la\to\infty}\xi^{+}(\la)=x_2.
\]
Hence, $\xi^{\pm}$ are bounded functions for $\la>0$, and $x_0=\sup\{|\xi^{\pm}(\la)|:\,\la>0\}>\max\{\rho, -x_1,x_2\}$. Therefore, the limit cycle $\Phi(\la)$ is contained in the strip $\{(x,y)\in\R^2:\,-x_0<x<x_0\}$ for every $\la>0$.

Now we prove that $\Phi(\la)$ intersects one of the straight lines $x=x_M$ or $x=x_m$. Note that the divergent of the vector field $X_{\la}(x,y)=(y,-x-\la f(x) y)$ is given by $\textrm{div} \: X_{\la}(x,y) = -\la f(x)$, for every $(x,y) \in \R^2$. Since $f(x) < 0$ for $x \in (x_M,x_m)$ then $\textrm{div} \: X_{\la}(x,y) > 0$ for $x \in (x_M,x_m)$ and $y \in \R$. Therefore, it follows from Bendixson's Criterion, that the limit cycle $\Phi(\la)$ must intersect one of the straight lines $x=x_M$ or $x=x_m$.
\end{proof}

\begin{figure}[h]
	\begin{overpic}[width=5.5cm]{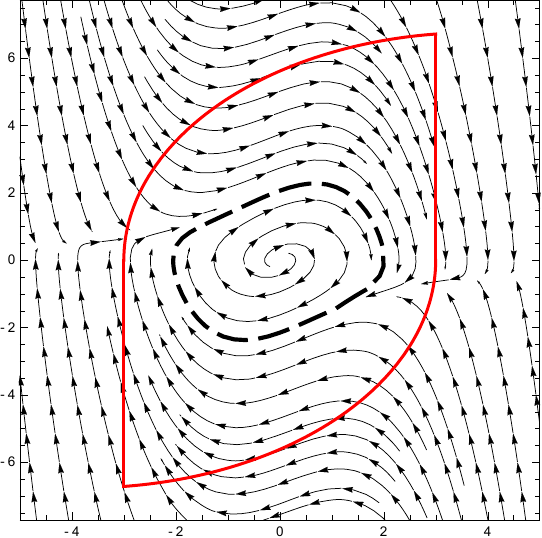}
	\end{overpic}
	\caption{\footnotesize Phase portrait of $X_{\la}$ in the van der Pol case (i.e. $f(x)=x^2-1$) assuming $\la=1/2$. The red line represents the boundary of the region $\R_{\la}$. In order to plot $\R_{\la}$ it was assumed that $x_0=3$, nevertheless the exactly value for $x_0$ is uknown. The dashed bold line represents the limit cycle.}
	\label{Rlambda}
\end{figure}

In order to build the region $R_{\la}$ we will need the following proposition.

\begin{proposition}\label{Rl}
Let $h_{\la}(x,y)=y-\gamma(x;\la)$, where $\gamma$ is defined in \eqref{gamma}, and let $X_{\la}(x,y)=(y,-x-\la f(x) y)$ be the vector field defined by system \eqref{S1}. Then, $\langle \nabla  h_{\la}(x,y),X_{\la}(x,y)\rangle<0$ for $y=\gamma(x;\la)$, $x\in(-x_0,x_0)$ and $\la>0$.
\end{proposition}
\begin{proof}

Firstly, let us assume that $\la\leq-1/(2m)$, where $m=\min\{f(x):\, -x_0\leq x\leq x_0\}$. In this case
\[
\ga(x;\la)=\sqrt{\dfrac{(x+x_0)(x_0-x-m\la(x+x_0))}{1+m\la}}
\]
and
\[
\langle \nabla  h_{\la}(x,y),X_{\la}(x,y)\rangle=-\la \gamma(x;\la)\left(f(x)-\dfrac{m x_0}{\sqrt{p(x)}}\right),
\]
where
\[
p(x)=-  (1 + m \la)^2 x^2 -2 m x_0 \la (1 + m \la)\, x + 
 x_0^2 (1 - m^2 \la^2).
\]
Note that $p(x)$ has a maximum at $\ov x=-m x_0\la/(1+ m\la)\in(-x_0,x_0)$ and $p(\ov x)=x_0^2$. Therefore
\begin{equation}\label{eq1}
\begin{array}{rl}
\langle \nabla  h_{\la}(x,y),X_{\la}(x,y)\rangle& \hspace{-.25cm} \leq-\la \gamma(x;\la)\left(f(x)-\dfrac{m x_0}{\sqrt{p(\ov x)}}\right)\\
& \hspace{-.25cm} =-\la \gamma(x;\la)\left(f(x)-m\right)<0.
\end{array}
\end{equation}

Now assume that $\la> -1/(2m)$. In this case 
\[
\ga(x;\la)=\sqrt{(x+x_0)((1+8m^2\la^2)x_0-x)},
\]
and
\[
\langle \nabla  h_{\la}(x,y),X_{\la}(x,y)\rangle=-\la \gamma(x;\la)\left(f(x)+\dfrac{4 m^2 \la x_0}{\sqrt{q(x)}}\right),
\]
where
\[
q(x)=-x^2 + 8 m^2 \la^2 x_0 \, x +  (1 + 8 m^2 \la^2)x_0^2.
\]
Note that $q(x)$ has a maximum at $\hat x=4 m^2 \la^2 x_0>x_0$. Hence, for $x\in(-x_0,x_0)$, $q$ reaches its maximum at $x_0$ and $q(x_0)=(4 m \la x_0)^2$. Therefore
\begin{equation}\label{eq2}
\begin{array}{rl}
\langle \nabla  h_{\la}(x,y),X_{\la}(x,y)\rangle& \hspace{-.25cm} \leq-\la \gamma(x;\la)\left(f(x)+\dfrac{4 m^2 \la x_0}{\sqrt{q(x_0)}}\right)\\
& \hspace{-.25cm} =-\la \gamma(x;\la)\left(f(x)-m\right)<0.
\end{array}
\end{equation}
From \eqref{eq1} and \eqref{eq2} we conclude this proof.
\end{proof}

The next result, which is an immediate consequence of Proposition \ref{Rl} and item (iv) of Theorem \ref{Thm1}, implies that $\Phi(\la)$ is contained in the interior of $R_{\la}=\{(x,y)\in\R^2:\, -x_0\leq x\leq x_0,\, -\ga(-x;\la)\leq y\leq \ga(x;\la)\}$ (see Figure \ref{Rlambda}). Note that the diameter of $R_{\la}$ is given by
\[
\begin{array}{rl}
\textrm{dim}(R_{\la})& \hspace{-.25cm} =d\big((-x_0,-\gamma(x_0;\la)),(x_0,\gamma(x_0;\la))\big) \vspace{.2cm}\\
& \hspace{-.25cm} =\left\{\begin{array}{ll}
		\dfrac{2x_0\sqrt{1+m^2\la^2}}{1+m\,\la}, & 0<\la\leq-\dfrac{1}{2 m},\\
		2x_0\sqrt{1+16 m^2 \la^2},& \la> -\dfrac{1}{2 m},\\
		\end{array}\right.
		\end{array}
\]
where $d$ denotes the usual metric of $\R^2$.

\begin{corollary}\label{RR}
For each $\la>0$, it holds that $\Phi(\la)\cap\{(x,\gamma(x;\la)):\,-x_0\leq x\leq x_0\}=\emptyset$ and $\Phi(\la)\cap\{(x,-\gamma(-x;\la)):\,-x_0\leq x\leq x_0\}=\emptyset$.
\end{corollary}

\section{Examples}

In this section, we illustrate the main result of the article with some examples. We start with the well known van der Pol equation. After that we will consider some examples where the function $f$ is not polynomial.

\subsection{The van der Pol equation} 

The van der Pol equation is a well known prototypical example where relaxation oscillations occur \cite{DurRou1996,Pol1926}. It is given by the equation \eqref{Eq1} with $f(x) = x^2-1$. For this function it is immediate to check that hypotheses {\bf (A1)} and {\bf (A2)} are fullfiled. In fact, we have that $x_M=-1$ and $x_m=1$ are the only zeros of $f$ with $f'(-1)=-2<0$ and $f'(1)=2>0$. The auxiliary function $F$ is given by $F(x) = x^3/3 - x$, and we see that the straight lines $y=F(x_M)=2/3$ and $y=F(x_m)=-2/3$, passing through the points $A = (x_M,F(x_M)) = (-1,2/3)$ and $B = (x_m,F(x_m))=(1,-2/3)$, intersect again the graphic of $F$ at the points $A' = (x_2,F(x_M)) = (2,2/3)$ and $B'= (x_1,F(x_m))=(-2,-2/3)$, respectively. Moreover, the constants $x_1^*$, $x_2^*$, $x^*$, and $r^*$ appearing in \eqref{xr} assume the values $-\sqrt{3}$, $\sqrt{3}$, $\sqrt{3}$, and $4$, respectively. Finally, it is well known that hypothesis {\bf (A3)} holds for the van der Pol equation.

Computing the function $\ov{F}$ given in  \eqref{Fbar} for the van der Pol equation we obtain $\ov{F}(r) = (r^2- 4)\pi /8.$
The unique positive zero of $\ov{F}$ is $\rho = 2$. Moreover, note that $2 \in (x^*,r^*)$, that is $2 \in (\sqrt{3},4)$. 
Therefore, from Theorem A, we conclude that, for every $\lambda > 0$, the van der Pol equation has a unique stable limit cycle $\Phi(\lambda)$. For sufficiently small values of $\lambda$ the cycle $\Phi(\lambda)$ approaches to the circle centered at the origin of radius $2$. We notice that this fact is known and fairly discussed in the literature (see, for instance, \cite{M62,SVM}). For sufficiently large values of $\lambda$, after the change of coordinates $P_{\lambda }(x,y)=\left(x,x^3/3-x+y/\lambda \right)$, the limit cycle $\Phi(\lambda )$ approaches a singular trajectory $\Gamma_0$.   

The van der Pol example can be generalized as follows. Let $f: \R \rightarrow \R$ be given by
$$f(x) = (x-x_M)(x-x_m)\bar{f}(x)$$ 
where $x_M<0<x_m$ and $\bar{f}: \R \rightarrow \R$ is a polynomial function satisfying $\bar{f}(x) > 0$ for every $x \in \R$.
Clearly, $f$ has precisely two real zeros, $x_M$ and $x_m$. Moreover, $f'(x_M) = (x_M-x_m)\bar{f}(x_M) < 0$ and $f'(x_m) = (x_m-x_M)\bar{f}(x_m) > 0$. Hence, the condition {\bf (A1)} is satisfied. 

Regarding the auxiliary function $F$, note that if $\bar{f}(x) = c > 0$ is constant, then $F$ is given by
$$F(x) = \int_0^x f(s)ds = c \bigg( \dfrac{x^3}{3} - (x_M+x_m)\dfrac{x^2}{2} + x_M x_m x \bigg).$$
Now if $\bar{f}$ is not constant, it takes the form
$$\bar{f}(x) =k \displaystyle{\prod_{j=1}^{n} \Big(x - (x_j+iy_j) \Big) \Big(x - (x_j-iy_j) \Big)},$$
with $k>0$. In this case the function $F$ is given by
$$F(x) = \int_0^x f(s)ds = k\dfrac{x^{2(n+1)+1}}{2n+3} + H(x),$$
where $H$ is a polynomial function of degree less than or equal to $2n+2$. 
In both cases ($\bar{f}$ constant or not), the following limits hold
$$
\lim_{x\to +\infty} F(x) = +\infty, \quad \textrm{and} \quad  \lim_{x\to -\infty} F(x) = -\infty,
$$
which imply the condition {\bf (A2)}. Therefore, assuming further hypothesis {\bf (A3)}, for the function $f$ like above, the conclusions of Theorem \ref{Thm1} apply. We emphasize that, in the Appendix section, sufficient conditions in order to guarantee hypothesis {\bf (A3)} are given. Particularly for the generalized van der Pol equation one could assume that $x_M=-x_m$.

\subsection{Non--polynomial examples}

In this subsection we provide examples of non--polynomial functions for which our hypotheses are fulfilled. We start with the following class of rational functions that generalizes the polynomial case
$$f(x) = \dfrac{(x-x_M)(x-x_m)p(x)}{q(x)},$$
where $x_M<0<x_m$, $p$ and $q$ are polynomial functions such that $p(x)q(x)>0$ for all $x \in \R$, and $\deg(q)< 2+\deg(p)$. Clearly, $f$ has precisely two real zeros, $x_M$ and $x_m$. Moreover, $f'(x_M) = (x_M-x_m)p(x_M)/q(x_M) < 0$ and $f'(x_m) = (x_m-x_M)p(x_m)/q(x_m) > 0$. Hence, the hypothesis {\bf (A1)} is satisfied. Since $\deg(q)< 2+\deg(p)$ and $p(x)q(x)>0$ for all $x \in \R$, then $\lim_{x \to \pm \infty} f(x) = +\infty$. This implies that
 $$
 \lim_{x\to +\infty} F(x) = \lim_{x\to +\infty} \int_0^x f(s)ds = \int_0^{+\infty} f(s)ds =  +\infty, 
 $$
 and
$$
\lim_{x\to -\infty} F(x) = \lim_{x\to -\infty} \int_0^x f(s)ds = -\lim_{x\to -\infty} \int_x^0 f(s)ds = -\int_{-\infty}^0 f(s)ds =  -\infty.
$$
In particular we can conclude that hypothesis {\bf (A2)} is fulfilled. Therefore, under the qualitative assumption {\bf (A3)}, Theorem \ref{Thm1} holds for the class of rational functions given above.

We also can find other examples of functions $f$, that are neither polynomial nor rational, such that our hypotheses are fulfilled. Below we list some of them:
$$
\begin{aligned}
f_1(x)= & \:  \exp(x)+\exp(-x)-b, \:\:\:\: \textrm{with} \:\: b>2, \\
f_2(x)= & \:  (2x^2-1)\exp(-x^2)+a, \:\: \textrm{with} \:\: 0<a<1.
\end{aligned}
$$
In fact, for the function $f_1$ we have that $x_M = \ln((b-\sqrt{b^2-4})/2)$ and $x_m = \ln((b+\sqrt{b^2-4})/2)$ are their unique zeros. Evaluating the derivative $f_1'$ in these two zeros gives $f_1'(x_M) = -\sqrt{b^2-4} < 0$ and $f_1'(x_m) = \sqrt{b^2-4} > 0$. Moreover, it can be checked that the auxiliary function $$F_1(x) = \int_0^x f_1(s)ds = \exp(x)-\exp(-x)-bx$$ satisfies
$$
\lim_{x\to +\infty} F_1(x) = +\infty, \quad \textrm{and} \quad  \lim_{x\to -\infty} F_1(x) = -\infty.
$$
Therefore, the conditions {\bf (A1)} and {\bf (A2)} are valid for the function $f_1$. 

With respect to the function $f_2$, we have that $x_M = -\sqrt{1/2-W(a\sqrt{e}/2)}$ and $x_m = \sqrt{1/2-W(a\sqrt{e}/2)}$ are their unique zeros, where $W$ is the Lambert function (principal branch), see \cite{Corless1996}. Note that since $0<a<1$, then $0 < W(a\sqrt{e}/2) < 1/2$. Evaluating $f_2'$ in these two zeros one obtains 
$$f_2'(x_M) = -\dfrac{a(1+W(a\sqrt{e}/2))\sqrt{2-4W(a\sqrt{e}/2)}}{W(a\sqrt{e}/2)} < 0$$
and 
$$f_2'(x_m) = \dfrac{a(1+W(a\sqrt{e}/2))\sqrt{2-4W(a\sqrt{e}/2)}}{W(a\sqrt{e}/2)} > 0.$$
The auxiliary function $F_2$ is given by 
$$F_2(x) = \int_0^x f_2(s)ds = x(a-\exp(-x^2)).$$
It is easy to check that $F_2$ satisfies
$$
\lim_{x\to +\infty} F_2(x) = +\infty, \quad \textrm{and} \quad  \lim_{x\to -\infty} F_2(x) = -\infty.
$$
Therefore, the conditions {\bf (A1)} and {\bf (A2)} are valid for the function $f_2$.

Note that for both functions $f_1$ and $f_2$, one has that $x_M=-x_m$. Hence, the condition {\bf (D2)} of Proposition \ref{ap:prop} of the Appendix section holds. Therefore, the condition {\bf (A3)} is valid for the functions $f_1$ and $f_2$. Consequently, the conclusions of Theorem \ref{Thm1} apply for these functions.

\section{Conclusion}\label{ConclusionSection}

In this paper, we consider one-parameter $\la>0$ families of Li\'{e}nard differential equations \eqref{Eq1} (equivalently the differential system \eqref{S1}). For each positive value of the parameter $\lambda$, we prove the existence of a limit cycle $\Phi(\la)$ as well as its continuous dependence on $\la$. We also provide the asymptotic behavior of such limit cycle for small and large values of $\la>0$.

For small values of the parameter $\la>0$, the differential system \eqref{S1} can be seen as a regular perturbation of the linear center $(y,-x)$. In this context the averaging theory provides useful tools for studying the birth of limit cycles from the periodic solutions of the center. When $\la$ assumes large values, the differential system \eqref{S1} can be converted into a slow--fast singularly perturbed system which can be treated with the techniques coming from the geometric singular perturbation theory.

Initially, the hypotheses {\bf (A1)} and {\bf (A2)} were assumed in order to guarantee that the manifold $\mathcal{S}$ was S-shaped (reflected) assuring then (Proposition \ref{PropRelaxationOscillation}) the existence of a relaxation oscillation approaching to the singular trajectory $\G_0$, when the parameter $\lambda$ takes large values. Surprisingly, the same hypotheses, {\bf (A1)} and {\bf (A2)}, also guaranteed the existence of a zero $\rho$ of the averaged function, satisfying some good properties on its Brouwer degree, which assured (Proposition \ref{av:prop}) the existence of a limit cycle approaching to the circle $S_\rho$, when $\la>0$ takes small values.

Under hypothesis {\bf (A3)} we were able to show that the limit cycle existing nearby the circle $S_{\rho}$, for small values of $\la>0$, is deformed continuously and increases its amplitude, when $\la$ becomes larger, approaching then to the singular trajectory $\G_0$. We also estimate the amplitude growth of the limit cycle.

\section*{Appendix: Uniqueness of the limit cycle}\label{AppendixSection}

First, recall that the existence of a limit cycle for system \eqref{S1} is a direct consequence of Dragil\"{e}v's Theorem \cite{Dragilev1952} (see Section 2). In addition, from \cite{V83}[Theorem 1], the existence of a limit cycle is obtained by assuming the continuity of $f$, $f(0)<0$, and that $f(x)$ and $xF(x)$ are positive for $|x|$ large enough. Thus, in this appendix, we provide some analytical conditions in order to fulfill hypothesis {\bf (A3)} of Theorem \ref{Thm1}: 
\begin{center}
{\it ``The differential system \eqref{S1} has at most one limit cycle.''} 
\end{center}
Accordingly, we shall prove the following result.

\begin{proposition}\label{ap:prop}
Assume hypotheses {\bf (A1)} and {\bf (A2)}. Then, the differential system \eqref{S1} has at most one periodic solution provided that at least one of the following conditions holds:
\begin{itemize}
\item[{\bf (D1)}] $F(\pm\infty)=\pm\infty$ and $x_1^*=-x_2^*$.
\item[{\bf (D2)}] $F(\pm\infty)=\pm\infty$ and $x_M=-x_m$.
\item[{\bf (D3)}] $f$ is nonincreasing in $(-\infty,0)$ and nondecreasing in $(0,+\infty)$.
\item[{\bf (D4)}] $F(x)/x$ is nonincreasing in $(-\infty,0)$ and nondecreasing in $[x_2^*,+\infty)$, and  $x_2^*\leq -x_1^*$.
\end{itemize}
\end{proposition}

The next two theorems are due to Sansone \cite{Sansone}. A proof for them can also be found in \cite{Zha1992}, see Theorems 4.2 and 4.3 of its Chapter 4.

\begin{theorem}[\cite{Sansone}]\label{Sansone1}
Consider the differential system \eqref{S1} and assume that $f$ is continuous and that $F(\pm\infty)=\pm\infty$. Suppose that there exist $\delta_{-1}<0<\delta_1$ and  $\Delta>0$ such that $f(x)<0$ for $x \in (\delta_{-1},\delta_1)$, $f(x)>0$ for $(-\infty,\delta_{-1}) \cup (\delta_1,\infty)$, and $F(\Delta)=F(-\Delta)=0$. Then, the differential system \eqref{S1} has a unique limit cycle, which is stable.
\end{theorem}

\begin{theorem}[\cite{Sansone}]\label{Sansone2}
Consider the differential system \eqref{S1} and assume that $f$ is continuous and that $F(\pm\infty)=\pm\infty$. Suppose that there exists $\delta>0$ such that $f(x)<0$ for $|x| < \delta$ and $f(x)>0$ for $|x|>\delta$. Then, the differential system \eqref{S1} has a unique limit cycle, which is stable.
\end{theorem}

The next theorem is due to Massera \cite{Massera}. A proof for it can also be found in \cite{Zha1992}, see Theorem 4.4 of its Chapter 4.

\begin{theorem}[\cite{Massera}]\label{Massera}
Consider the differential system \eqref{S1} and assume that $f$ is continuous. Suppose that there exist $\delta_{-1}<0<\delta_1$ such that $f(x)<0$ for $x \in (\delta_{-1},\delta_1)$ and $f(x)>0$ for $(-\infty,\delta_{-1}) \cup (\delta_1,\infty)$, and that  $f$ is nonincreasing in $(-\infty,0)$ and nondecreasing in $(0,+\infty)$. Then, the differential system \eqref{S1} has a unique limit cycle, which is stable.
\end{theorem}

The next theorem is due to Figueiredo \cite{Figueiredo}. A proof for it can also be found in \cite{Ye}, see Theorem 6.9.

\begin{theorem}[\cite{Figueiredo}]\label{Figueiredo}
Consider the differential system \eqref{S1} and assume that $F(x)\not\equiv0$ in a neighborhood of the origin. Suppose that there exists $\delta>0$ such that $x\,F(x)\leq0$ for $|x|\leq\delta$, $F(x)\geq0$ for $x>\delta$, and $F(x)/x$ is nonincreasing in $(-\infty,0)$ and nondecreasing in $[\delta,+\infty)$. Then, the differential system \eqref{S1} has at most one limit cycle.
\end{theorem}

Now we prove Proposition \ref{ap:prop}.

\begin{proof}[Proof of Proposition \ref{ap:prop}]
Item {\bf (D1)} is consequence of Theorem \ref{Sansone1} if we take $\delta_{-1}=x_M$, $\delta_1=x_m$, and $\Delta=-x_1^*=x_2^*$. Item {\bf (D2)} is consequence of Theorem \ref{Sansone2} if we take $\delta=-x_M=x_m$. Item {\bf (D3)} is consequence of Theorem \ref{Massera} if we take $\delta_{-1}=x_M$ and $\delta_1=x_m$. Item {\bf (D4)} is consequence of Theorem \ref{Figueiredo} if we take $\delta=x_2^*$.
\end{proof}

For more analytical conditions ensuring the uniqueness of limit cycles in Li\'{e}nard differential systems we cite, for instance, the books \cite{Ye,Zha1992},  the works \cite{MH1997,CarVil2005} where the authors relaxed the symmetry assumption $x_1^*=-x_2^*$ of the Sansone's Theorem \ref{Sansone1}, and also the works \cite{VilZan2016,HVZ2018,VilZan2018}. 

\section*{Acknowledgements}

We are grateful to the anonymous referee for valuable comments and suggestions.

P.T.C. is partially supported by Fundação de Amparo à Pesquisa do Estado de São Paulo (FAPESP) grants 2019/00976-4 and 2013/24541-0. D.D.N is partially supported by Fundação de Amparo à Pesquisa do Estado de São Paulo (FAPESP) grant 2018/16430-8, and by Conselho Nacional de Desenvolvimento Científico e Tecnológico (CNPq) grants 306649/2018-7 and 438975/2018-9. Both authors are partially supported by Coordenação de Aperfeiçoamento de Pessoal de Nível Superior (CAPES), Program CSF-PVE, grant 88881.030454/2013-0.

\bibliographystyle{abbrv}
\bibliography{biblio.bib}

\end{document}